                    \def\version{April 1, 2015}                       %

 \documentclass[reqno,11pt]{amsart}
 \usepackage{amsmath, amsthm, a4, latexsym, amssymb}
\usepackage[unicode]{hyperref}
\usepackage{srcltx}
\usepackage{xcolor}

\def\@rmrk#1#2{\refstepcounter
    {#1}\@ifnextchar[{\@yrmrk{#1}{#2}}{\@xrmrk{#1}{#2}}}

\makeatletter\@addtoreset{equation}{section}\makeatother

 \newfont{\bfit}{cmbxti10 scaled 1200}

\renewcommand{\d}{{\rm d}}

 \newcommand{\e}{{\rm e} }

 \newcommand{\eps}{\varepsilon}

 \newcommand{\R}{\mathbb{R}}
 \newcommand{\N}{\mathbb{N}}
 \newcommand{\Z}{\mathbb{Z}}

 \newcommand{\E}{\mathbb{E}}
 \renewcommand{\P}{\mathbb{P}}
 \def\1{{\mathchoice {1\mskip-4mu\mathrm l} 
{1\mskip-4mu\mathrm l}
{1\mskip-4.5mu\mathrm l} {1\mskip-5mu\mathrm l}}}

 \newcommand{\Mcal}{{\mathcal M}}

\newcommand{\heap}[2]{\genfrac{}{}{0pt}{}{#1}{#2}}

\newcommand{\ssup}[1] {{\scriptscriptstyle{({#1}})}}

\renewcommand{\subsection}{\secdef \subsct\sbsect}
\newcommand{\subsct}[2][default]{\refstepcounter{subsection}
\vspace{0.15cm}
{\flushleft\bf \arabic{section}.\arabic{subsection}~\bf #1  }
\nopagebreak\nopagebreak}
\newcommand{\sbsect}[1]{\vspace{0.1cm}\noindent
{\bf #1}\vspace{0.1cm}}

{\nopagebreak {\hfill\rule{2mm}{2mm}}\\ }

\newtheorem{theorem}{Theorem}[section]
\newtheorem{lemma}[theorem]{Lemma}
\newtheorem{cor}[theorem]{Corollary}
\newtheorem{prop}[theorem]{Proposition}

\newtheoremstyle{thm}{1.5ex}{1.5ex}{\itshape\rmfamily}{}
{\bfseries\rmfamily}{}{2ex}{}

\newtheoremstyle{rem}{1.3ex}{1.3ex}{\rmfamily}{}
{\itshape\rmfamily}{}{1.5ex}{}
\theoremstyle{rem}
\newtheorem{remark}{{\slshape\sffamily Remark}}[]

\refstepcounter{subsubsection}

\def\thebibliography#1{\section*{References}
  \list%
  {\arabic{enumi}.}
    {\settowidth\labelwidth{[#1]}\leftmargin\labelwidth
    \advance\leftmargin\labelsep
    \parsep0pt\itemsep0pt
    \usecounter{enumi}}
    \def\newblock{\hskip .11em plus .33em minus .07em}
    \sloppy                   
    \sfcode`\.=1000\relax}

\begin{document}
\title[Quenched large deviations for simple random walk on supercritical percolation clusters]
{\large Quenched large deviations for simple random walk on supercritical percolation clusters}
\author[Noam Berger and Chiranjib Mukherjee]{}
\maketitle
\thispagestyle{empty}
\vspace{-0.5cm}

\centerline{\sc Noam Berger \footnote{Hebrew University Jerusalem and Technical University Munich, Boltzmannstrasse 3. Garching (near Munich), {\tt noam.berger@tum.de}} and
Chiranjib Mukherjee \footnote{Technical University Munich, Boltzmannstrasse 3. Garching (near Munich), {\tt chiranjib.mukherjee@tum.de}}}
\renewcommand{\thefootnote}{}
\footnote{\textit{AMS Subject
Classification:} 60J65, 60J55, 60F10.}
\footnote{\textit{Keywords:} Random walk, percolation clusters,
large deviations}

\vspace{-0.5cm}
\centerline{\textit{Hebrew University Jerusalem and Technical University Munich}}
\vspace{0.2cm}

\begin{center}
\version
\end{center}

\begin{quote}{\small {\bf Abstract:} 
We prove a {\it{quenched}} large deviation principle (LDP) for a simple random walk on a supercritical percolation cluster on $\Z^d$, $d\geq 2$. We take the point of view of the moving particle and first prove
 a quenched LDP for the distribution of the {\it{pair empirical measures}} of the environment Markov chain. Via a contraction principle, this reduces easily to a quenched LDP for the distribution of the mean velocity of the 
 random walk and both rate functions admit explicit (variational) formulas. Our results are based on invoking ergodicity arguments in this non-elliptic set up to control the growth of {\it{gradient functions (correctors)}} which come up naturally via convex variational analysis in the context of homogenization of random Hamilton Jacobi Bellman equations along the arguments of Kosygina, Rezakhanlou and Varadhan (\cite{KRV06}). Although enjoying some similarities, our gradient function is structurally different from {\it{the}} classical {\it{Kipnis-Varadhan corrector}}, a well-studied object in the context of reversible random motions in random media.
 }

 

\end{quote}
\section{Introduction and main results}
\subsection{The model.}


We consider a simple random walk on the infinite cluster of a supercritical bond percolation on $\Z^d$, $d\geq 2$. Conditional
on the event that the origin lies in the infinite open cluster, it is known that a law of large numbers 
and quenched central limit theorem hold, see Sidoravicius and Sznitman (\cite{SS04}) for $d\geq 4$ and Mathieu- Piatnitski (\cite{MP07}) and Berger- Biskup (\cite{BB07})
for any $d\geq 2$. However, treatment of such standard questions for this model needs care because of its inherent {\it{non-ellipticity}}-- a problem which permeates in several forms in the above mentioned literature.

Questions on quenched large deviations for general random motions in random environments have also been studied (see section 2.1 for a detailed review on the existing literature) in a fundamental work of Kosygina-Rezakhanlou-Varadhan (\cite{KRV06}) for a diffusion with random drift and in related work of Rosenbluth (\cite{R06}) and Yilmaz (\cite{Y08}) for general random walks in random environments. Other results of relevance are by Rassoul-Agha, Sepp\"al\"ainen and Yilmaz (see \cite{RSY13}) on directed, undirected and stretched polymers in a random (and possibly unbounded) potential. All these results, however, require certain {\it{moment conditions}} on the environment ({\it{ellipticity}}) which needs to be necessarily dropped when studying a non-elliptic model, for example, the classical simple random walk on the supercritical percolation cluster (SRWPC). In this context, it is the goal of the present article to study quenched large deviation principles for the distribution of the empirical measures of the environment Markov chain of SRWPC ({\it{level- 2}}) and subsequently deduce the particle dynamics of the rescaled location ({\it{level -1}}) of the walk on the cluster. We start with the precise description of the model of classical bond percolation on $\Z^d$.

 We fix $d\geq 2$ and denote by $\mathbb B_d$ the set of nearest neighbor
edges of the lattice $\Z^d$ and by $\mathbb U_d$ the set of edges from the origin to its nearest neighbor. Let $\Omega= \{0,1\}^{\mathbb B_d}$ be the space of all {\it{percolation configurations}} $\omega=(\omega_b)_{b\in\mathbb B_d}$. In other words, 
$\omega_b=1$ refers to the edge $b$ being {\it{present}} or {\it{open}}, while $\omega_b=0$ 
implies that it is {\it{vacant}} or {\it{closed}}. Let $\mathcal B$ be the Borel-$\sigma$-algebra on $\Omega$ defined by the product topology. We fix the {\it{percolation parameter}} $p\in(0,1)$ and denote by $\P=\P_p:=\big(p \delta_1+ (1-p) \delta_0\big)^{\mathbb B_d}$
the product measure with marginals $\P(\omega_b=1)=p=1- \P(\omega_b=0)$. Note that 
$\Z^d$ acts as a group on $(\Omega, \mathcal B, \P)$ via translations. In other words, for each $x\in \Z^d$, $\tau_x: \Omega \longrightarrow \Omega$ acts as a {\it{shift}} given by $(\tau_x\omega)_b= \omega_{x+b}$.
Note that the product measure $\P$ is invariant under this action.

For each $\omega\in\Omega$, let $\mathcal C_\infty(\omega)= \{x\in \Z^d\colon\, x\longleftrightarrow \infty\}$ denote the set of points $x\in \Z^d$, which finds
an infinite self-avoiding path using occupied bonds in the configuration $\omega$. It is known that there is a critical percolation  probability $p_c=p_c(d)$ which is the infimum of all $p$'s such that $\P(0\in \mathcal C_\infty)>0$. In this paper we only consider the case $p>p_c$.

For $p>p_c$, the set $\mathcal C_\infty(\omega)$ is $\P$-almost surely non-empty and connected. 
Let $\Omega_0=\{ 0\in \mathcal C_\infty\}$.
For $p>p_c$ we define the conditional probability $\P_0$ by
$$
\P_0(A)= \P\big(A\big|\Omega_0\big) \qquad A\in\mathcal B.
$$
We now define a (discrete time) simple random walk on the supercritical percolation cluster $\mathcal C_\infty$ as follows. Let the walk start at the origin and at each unit of time,
the walk moves to a nearest neighbor site chosen uniformly at random from the accessible neighbors. More precisely, for each $\omega\in \Omega_0$, $x\in \Z^d$ and $e\in \mathbb U_d$, we set
\begin{equation}\label{pidef}
\pi_\omega(x,e)=\frac{ \1_{\{\omega_e=1\}} \circ \tau_x}{\sum_{ |e^\prime|=1} \1_{\{\omega_{e^\prime}=1\}} \circ \tau_x} \in [0,1],
\end{equation}
and define a simple random walk $X=(X_n)_{n\geq 0}$ as a Markov chain taking values in $\Z^d$ with the transition probabilities
\begin{equation}\label{pitransit}
\begin{aligned}
&P^{\pi,\omega}_{0}(X_0=0)=1,\\
&P_{0}^{\pi,\omega}\big(X_{n+1}= x+e\big| X_n=x\big)=\pi_\omega(x,e).
\end{aligned}
\end{equation}
This is a canonical way to ``put" the Markov chain on the infinite cluster $\mathcal C_\infty$ and henceforth, we refer to this Markov chain as the {\it{simple random walk on the percolation cluster}} (SRWPC). 

\subsection{Main results: Quenched large deviation principle.}

For each $\omega\in \Omega_0$, we consider the process $(\tau_{X_n}\omega)_{n\geq 0}$ which is a Markov chain taking values in the space of environments
$\Omega_0$. This is the {\it{environment seen from the particle}} and plays an important role in the present context, see section 3.1 for a detailed description.
We denote by
\begin{equation}\label{localtime}
\mathcal L_n= \frac 1n \sum_{n=0}^{n-1} \delta_{\tau_{X_k}\omega,{X_{k+1}-X_{k}}}
\end{equation}
the empirical measure of the environment Markov chain and the nearest neighbor steps of the SRWPC $(X_n)_{n\geq 0}$. This is a random element of $\Mcal_1(\Omega_0 \times  \mathbb U_d)$, the space of probability measures on $\Omega_0\times \mathbb U_d$, which is compact when equipped with the weak topology (note that, $\Omega_0 \subset \Omega$ is closed and hence compact). 

We note that, via the mapping $(\omega,e)\mapsto (\omega, \tau_e\omega)$ the space $\Mcal_1(\Omega_0 \times  \mathbb U_d)$ is embedded into $\Mcal_1(\Omega_0\times \Omega_0)$, 
and hence, any element $\mu\in \Mcal_1(\Omega_0\times \mathbb U_d)$ can be thought of as the {\it{pair empirical measure}} of the environment Markov chain. In this terminology,
we can define its marginal distributions by 
\begin{equation}\label{marginals}
\begin{aligned}
&\d (\mu)_1(\omega)= \sum_{e\in \mathbb U_d} \d \mu(\omega,e), \\
& \d (\mu)_{2}(\omega)= \sum_{e\colon \, \tau_e\omega^\prime=\omega} \d \mu(\omega^\prime,e)=\sum_{e\in \mathbb U_d} \d \mu(\tau_{-e}\omega,e).
\end{aligned}
\end{equation}
A relevant subspace of $\Mcal_1(\Omega_0 \times  \mathbb U_d)$ is given by  
\begin{equation}\label{relevantmeasures}
\begin{aligned}
\Mcal_1^\star=\Mcal_{1}^ {\star}(\Omega_0\times \mathbb U_d)&=\bigg\{\mu\in\Mcal_1(\Omega_0 \times  \mathbb U_d)\colon\, (\mu)_1=(\mu)_ 2\ll \P_0 \, \,\mbox{and}\, \,\P_0\mbox{- almost surely,}\\
&\qquad\qquad\frac{\d \mu(\omega,e)}{\d (\mu)_{ 1}(\omega)} >0 \,\,
\mbox{if and only if}\,\,\omega(e)=1\,\mbox{for}\, e\in \mathbb U_d\bigg\}.
\end{aligned}
\end{equation}
A simple, though important, relation between the space $\Mcal_1^\star$ and the environment process $(\tau_{X_n}\omega)_{n\geq 0}$ 
is made transparent in section 3.1-- elements in $\Mcal_1^\star$
are in one-to-one correspondence to Markov kernels on $\Omega_0$ which admit {\it{invariant probability measures}} which are {\it{absolutely continuous}} with respect to $\P_0$, see Lemma \ref{onetoone}.

Finally, we define a {\it{relative entropy functional}} $\mathfrak I: \Mcal_1(\Omega_0 \times  \mathbb U_d) \rightarrow [0,\infty]$ via
\begin{equation}\label{Idef}
\mathfrak I(\mu) =
\begin{cases}
\int_{\Omega_0} \d \P_0 \sum_{e\in \mathbb U_d} \d\mu(\omega,e) \log \frac{\d \mu(\omega,e)}{\d(\mu)_{1}(\omega) \pi_\omega(0,e)} \quad\mbox{if}\,\mu\in\Mcal_1^\star,
\\
\infty\qquad\qquad\qquad\qquad\qquad\qquad\qquad\qquad\quad\mbox{else.}
\end{cases}
\end{equation}
For every continuous, bounded and real valued function $f$ on $\Omega_0 \times  \mathbb U_d$, we denote by
$$
\mathfrak I^\star(f)= \sup_{\mu\in\Mcal_1(\Omega_0 \times  \mathbb U_d)} \big\{ \langle f,\mu\rangle - \mathfrak I(\mu)\big\}
$$
the {\it{Fenchel-Legendre transform}} of $\mathfrak I(\cdot)$ and by $\mathfrak I^{\star\star}(\mu)$, for any $\mu\in \Mcal_1(\Omega_0 \times  \mathbb U_d)$, the {\it{Fenchel-Legendre transform}} of $\mathfrak I^\star(\cdot)$. 

We are now ready to state the main result of this paper, which proves a large deviation principle for the distributions $\P^{\pi,\omega}_0 \mathcal L_n^{-1}$ and $\P^{\pi,\omega}_0 {\frac {X_n}{n}}^{-1}$ 
on $\Mcal_1(\Omega_0\times\mathbb U_d)$ and $\R^d$ respectively. Both statements hold true for $\P_0$- almost every $\omega\in \Omega_0$. In other words our results concern {\it{quenched large deviations}} and
share close analogy to the results by Rosenbluth (\cite{R06}) and Yilmaz (\cite{Y08}) for random works in random environments. In the present context, due to zero
transition probabilities of the SRWPC, we necessarily have to drop
their assumption requiring $p$-th moment of the logarithm of the random walk transition probabilities being finite, for $p>d$.
Here is the statement of our main result.
\begin{theorem}[Quenched LDP for the pair empirical measures]\label{thmlevel2}

Let $d\geq 2$ and $p>p_c(d)$. Then for $\P_0$- almost every $\omega\in \Omega_0$, the distributions  
of $\mathcal L_n$ under $P^{\pi,\omega}_0$
 satisfies a large deviation principle in the space of probability measures on
$\Mcal_1(\Omega_0 \times  \mathbb U_d)$ equipped with the weak topology.
The rate function $\mathfrak I^{\star\star}$ is the double Fenchel-Legendre transform of the functional $\mathfrak I$. Furthermore, $\mathfrak I^{\star\star}$ is convex and has compact level sets.
\end{theorem}
In other words, for every closed set $\mathcal C\subset \Mcal_1(\Omega_0 \times  \mathbb U_d)$, every open set $\mathcal G\subset \Mcal_1(\Omega_0 \times  \mathbb U_d)$ and $\P_0$- almost every $\omega\in \Omega_0$,
\begin{equation}\label{ldpub}
\limsup_{n\to\infty} \frac 1n \log \P^{\pi,\omega}_0 \big(\mathcal L_n\in \mathcal C\big) \leq -\inf_{\mu\in \mathcal C} \mathfrak I^{\star\star}(\mu),
\end{equation}
and 
\begin{equation}\label{ldplb}
\limsup_{n\to\infty} \frac 1n \log \P^{\pi,\omega}_0 \big(\mathcal L_n\in \mathcal G\big) \geq -\inf_{\mu\in \mathcal G} \mathfrak I^{\star\star}(\mu).
\end{equation}
\begin{remark}
The functional $\mathfrak I$ is convex on $\Mcal_1(\Omega_0\times\mathbb U_d)$, but fails to be lower semicontinuous, see Lemma \ref{nonlsc}. Hence, $\mathfrak I^{\star\star}\ne \mathfrak I$.
\end{remark}

Theorem \ref{thmlevel2} is an easy corollary to the existence of the limit
$$
\lim_{n\to\infty} \frac 1n \log E^{\pi,\omega}_0 \big\{\exp\{n\big \langle f, \mathcal L _n\big\rangle\big\}\big\}
=\lim_{n\to\infty} \frac 1n \log E^{\pi,\omega}_0 \bigg\{\exp\bigg( \sum_{k=0}^{n-1}f\big(\tau_{X_k}\omega, X_k-X_{k-1}\big)\bigg)\bigg\},
$$
for every continuous, bounded function $f$ on $\Omega_0 \times  \mathbb U_d$ and the symbol $\langle f, \mu\rangle$ denotes, in this context, the integral $\int_{\Omega_0} \d \P_0 (\omega) \sum_{e\in \mathbb U_d} f(\omega,e) \d \mu(\omega,e)$.
We formulate it as a theorem.

\begin{theorem}[Logarithmic moment generating functions]\label{thmmomgen}
For $d\geq 2$, $p> p_c(d)$ and every continuous and bounded function $f$ on $\Omega_0 \times  \mathbb U_d$,
$$
\lim_{n\to\infty} \frac 1n \log E^{\pi,\omega}_0 \bigg\{\exp\bigg( \sum_{k=0}^{n-1}f\big(\tau_{X_k}\omega, X_k-X_{k-1}\big)\bigg)\bigg\} = \sup_{\mu\in \Mcal_{1}^{\star}} \big\{\langle f,\mu\rangle- \mathfrak I(\mu)\big\} \quad\P_0-\mbox{a.s.}
$$
\end{theorem}
We will first prove Theorem \ref{thmmomgen} and deduce Theorem \ref{thmlevel2} directly.

Note that via the contraction map $\xi: \Mcal_1(\Omega_0 \times  \mathbb U_d) \longrightarrow \R^d$, 
$$
\mu\mapsto \int_{\Omega_0} \sum_e \, e\,\d\mu(\omega,e),
$$
we have $\xi(\mathcal L_{n})= \frac{X_n-X_0}n= \frac {X_n} n$. Our second main result is the following corollary to Theorem \ref{thmlevel2}.
\begin{cor}[Quenched LDP for the mean velocity of SRWPC]\label{thmlevel1}
Let $d\geq 2$ and $p>p_c(d)$. Then the distributions $P^{\pi,\omega}_0\big(\frac {X_n}n\in \cdot\big)$ satisfies a large deviation principle 
with a rate function 
\begin{equation}\label{level1rate}
\begin{aligned}
 J(x)&= \inf_{\mu\colon \xi(\mu)=x} \mathfrak I(\mu)\qquad x\in\R^d.
\end{aligned}
\end{equation}
\end{cor}
\begin{remark} Note that Corollary \ref{thmlevel1} has been obtained by Kubota (\cite{K12}) for the SRWPC (see also the results of J.C. Mourrat (\cite{M12}) for similar work in the context of random walks in random potential), though with transition probabilities slightly different from ours (Kubota
considered a random walk on the cluster which picks a neighbor at random and if the corresponding
edge is occupied, the walk moves to its neighbor. If the edge is vacant, the move is suppressed, i.e., the walk is {\it{lazy}}. Note that in our model the random walk
takes no pauses, i.e., the random walk is {\it{agile}}). However, his method of proof as well as the description 
of the rate function $J$ is completely different from ours. See section 2.1 for a comparison of results and proof techniques.
\end{remark}


\subsection{Literature review} 


In $d=1$, Greven and den Hollander (\cite{GdH94}) derived a quenched large deviation principle for the mean velocity of a random walk in i.i.d. random environment based on techniques
from branching processes and obtained a formula for the rate function. Indepedently, using passage times on $\Z$, Comets, Gantert and Zeitouni (\cite{CGZ00}) derived the LDP (also in the {\it{annealed}} and {\it{functional}} form)
for stationary and ergodic environments. For $d\geq 1$, Zerner (\cite{Z98}, see also Sznitman (\cite{S94}) for Brownian motion in a Poissonian potential) proved a quenched LDP under the assumption that
$-\log \pi(0,e)$ has finite $d$-th moment and that there exists arbitrary large regions in the lattice where the local drifts $\sum_e e \pi(0,e)$ ``points towards the origin'' (the {\it{nestling property}}).  His method is based on proving 
shape theorems and deriving the LDP, the driving force here being the sub-additive ergodic theorem. Invoking the sub-additivity more directly, Varadhan (\cite{V03}) 
proved a quenched LDP without assuming the nestling assumption, but for uniformly elliptic environments. 

Kosygina, Rezakhanlou and Varadhan (\cite{KRV06}) derived a novel method for proving quenched LDP using the {\it{environment seen from the particle}} in the context
of a diffusion with a random drift assuming some regularity conditions on the drift.
This method goes parallel to {\it{quenched homogenization of random Hamilton- Jacobi- Bellman equations}}, see section \ref{HJB} for a review. Rosenbluth (\cite{R06}) 
invoked this theory to multidimensional random walks in random environments and obtained a rate function given by the dual
of the {\it{effective Hamiltonian}} which admits a {\it{formula}}. The regularity assumption (\cite{KRV06}) on the Hamiltonian under which homogenization takes place (or quenched large deviation principle
hold) now translates to the assumption that $-\log \pi (0,e)$ has finite $d+ \eps$ moment, $\eps>0$, of Rosenbluth (\cite{R06}). Under this moment assumption, Yilmaz extended this to 
a level- 2 quenched LDP (as in Theorem \ref{thmlevel2}) and subsequently Rassoul-Agha and Sepp\"al\"ainen (\cite{RS11}) proved a level-3 (process level) LDP getting variational formulas for the corresponding
rate functions. This method has been further exploited for studying quenched LDP and {\it{free energy}} for (directed and non-directed) random walks in a random potential $V=- \log \pi$, see Rassoul-Agha, Sepp\"al\"ainen and Yilmaz (\cite{RSY13}) and results concerning {\it{log-gamma polymers}} of Georgiou, Raggoul-Agha, Sepp\"al\"ainen and Yilmaz (\cite{GRSY13}) (see also \cite{RSY14} and \cite{GRS14} for related models).
All these results, though seeing significant achievements, work only under the standing assumption $V=-\log \pi \in L^p(\P)$ for $p>d$ and do not cover the case $V= \infty$, pertinent to the case of a random walk on a supercritical
percolation cluster we are interested in. 

As mentioned before, Kubota (\cite{K12}), based on the method of Zerner ({\cite{Z98}) (see also Mourrat (\cite{M12}) proved a quenched LDP for the mean velocity of the random walk $\frac {X_n} n$
on a supercritical percolation cluster, which is very close to Corollary \ref{thmlevel1}. He also used sub-addtivity and overcame the lack of the moment criterion of Zerner
by using classical results about the geometry of the percolation. This way he obtained a rate function which is convex and is the given by the Legendre transform of the {\it{Lyapunov exponents}} derived by Zerner {\cite{Z98}). 
However, using the sub-additive ergodic theorem one does not get a satisfactory {\it{formula}} for the rate function, nor does this method seem amenable for deriving
a {\it{level 2}} quenched LDP as in Theorem \ref{thmlevel2}. 

\subsection{Outline.}

Let us now turn to a sketch of the proof of Theorem \ref{thmmomgen} whose guiding philosophy is based on the ideas of \cite{KRV06}. 
A rough idea is to first work on the space of environments to obtain certain ergodic properties to derive the lower bound, 
which, by variational techniques can be shown to be an upper bound as well, provided that a certain class of gradient functions, which show up naturally in the variational analysis, have a sub-linear growth at infinity. Controlling this growth (which is crucial for the upper bound) requires certain regularity properties (see \eqref{diffassump}) of the random drift (or the moment conditions of \cite{R06} and \cite{Y08}) which we necessarily drop for the SRWPC
and prove the sub-linearity of the gradient functions using ergodicity and geometric properties of percolation. An upshot of the aforementioned variational analysis for our case is the boundedness of the gradient functions on the cluster (our {\it{effective Hamiltonian}} does not grow), a key information for proving their sub-linear growth property.


We end this section summarizing the organization of the rest of the article. Section 2 reviews the approach of \cite{KRV06} shortly. Although this is well-known, we underline
its main ideas to make our proof a bit more transparent and put it into context. Sections 3, 4 and 5 are devoted to the proof of Theorem \ref{thmmomgen}: 
Section 3 proves the lower bound using some ergodicity arguments of Markov chains on environments, where arbitrary Markov kernels possibly admit zero transition probabilities. Section 4 introduces
a certain class of {\it{gradients}} or {\it{correctors}} (this should not be confused with the classical {\it{Kipnis- Varadhan corrector}}, see Remark \ref{KVremark} below), which
 admit a sub-linear growth on the cluster at infinity, which is the main step for deriving the upper bound in section 5. Theorem \ref{thmlevel2} and Corollary \ref{thmlevel1} are easily deduced from Theorem \ref{thmmomgen} in section 6.
 
 \begin{remark}\label{KVremark}
In our proof, as mentioned, a class $\mathcal G_\infty$ of gradient functions show up naturally in the context of variational analysis (see section 4 and section 5.2). These objects
 share close similarities to {\it{the Kipnis- Varadhan corrector}} which is a central object of interest for reversible random motions in random media. Particularly for SRWPC this is crucial for proving a central limit
 theorem (\cite{SS04}, \cite{MP07}, \cite{BB07})-- the corrector expresses the distance between the random walk and a (harmonic) embedding of the cluster in $\R^d$ where the random walk becomes a martingale. 
 Finer quantitative questions (for example, existence of all moments) are of interest, see related work of Lamacz, Neukamm and Otto (\cite{LNO13}) on a similar model
 of percolation with all bonds parallel to the direction $e_1$ being declared open. However, our functions in class $\mathcal G_\infty$ are structurally different from the Kipnis- Varadhan corrector. Though
 they share similar properties as {\it{gradients}}, objects in class $\mathcal G_\infty$, in particular, miss the above mentioned 
 {\it{harmonicity}} of the Kipnis- Varadhan corrector. Large deviation (lower) bounds are based on a certain {\it{tilt}} which spoils the inherent {\it{reversibility}} of the model, a crucial base of Kipnis- Varadhan theory.
 \end{remark}


\section{Input from quenched homogenization: Main idea of our proof}\label{HJB}


As mentioned, we follow \cite{KRV06} and work with a diffusion in random drift. 
The notation used in this section should be treated independently from other parts of this article.
Let $(\Omega, \mathcal F, \P)$ be a probability space on which  $\R^d$ acts as an additive group of translations. Let 
$b: \Omega \longrightarrow \R^d$ be a nice vector field and let us consider a random diffusion $(X_t)_{t\geq 0}$ on $\R^d$ whose infinitesimal generator
is given by 
$$
\big(\mathcal A^{\ssup \omega}_b u\big)(x)= \frac 12 \Delta u(x)+ \big\langle b_\omega(x), \nabla u(x)\big\rangle,
$$
where the drift $b_\omega(x)= b(\tau_x \omega)$ is generated by the action of $\{\tau_x\}_{x\in \R^d}$ on $\omega$. Let $P^{b, \omega}_0$ be the corresponding
Markovian measure starting at $0\in \R^d$ at time $0$ for each $\omega\in \Omega$. We would like to have a large deviation principle
for the distribution $P^{b, \omega}_0 \big( \frac {X(t)} t\in \cdot\big)$ almost surely with respect to $\P$.

The first step is to ``lift" the diffusion $(X_t)_{t\geq 0}$ to the {\it{environment process}}
$\overline\omega_t=\tau_{X_t}(\omega)$ 
taking values on $\Omega$ with infinitesimal generator $\overline{\mathcal A}_{b}= \frac 12 \Delta+ b \cdot \nabla$
where $\nabla= \{\nabla_i\}_{i=1}^d$ is the gradient given by the generators
$$
\big(\nabla_i u\big)(\omega)= \lim_{\eps\to 0} \frac {u(\tau_{\eps \mathbf{e_i}} \omega)- u(\omega)} \eps \quad i=1,\dots, d,
$$
of the translation semigroup $\{\tau_x\}_{x\in \R^d}$. A key step is to find a probability measure which is invariant under $\overline {\mathcal A}$ 
and absolutely continuous with respect to $\P$, i.e., to find $\phi\in L^1(\P)$ with $\phi\geq 0, \, \int \phi \d \P=1$ such that $\phi$ satisfies 
$$
\frac 12 \Delta \phi= \nabla \cdot (b \phi).
$$
If we fancy that such a function $\phi$ exists, then the measure $\phi \d \P$ is ergodic (see Kozlov(\cite{K85}) and Papanicolau- Varadhan (\cite{PV81})). Hence, by the ergodic theorem,
$$
\lim_{t\to\infty} \frac 1t \int_0^t f(\overline \omega(s)) \d s= \int f(\omega) \phi(\omega) \d \P \qquad \P- \mbox{ almost surely},
$$
for any test function $f$ on $\Omega$. This also translates to an ergodic theorem on $\R^d$ for the stationary process $g(\omega, X_t)= f(\tau_{X_t}\omega)$:
$$
 \lim_{t\to\infty} \frac 1t \int_0^t g(\omega, X_s) \d s= \int f(\omega) \phi(\omega) \d \P \qquad P^{b,\omega}_0 - \mbox{ almost surely}.
 $$
If we write down the martingale problem $X_t= W_t+ \int_0^t b(\omega, X_s) \d s$,
where $(W_t)_{t\geq 0}$ is a standard Brownian motion, then we have the law of the large numbers 
 $$
\lim_{t\to \infty} \frac {X_t} t = \lim_{t\to\infty} \int_0^t b(\omega, X_s) \d s = \int b(\omega) \phi(\omega) \d \P,
$$
almost surely  with respect to $\P$ and $P^{b,\omega}$.

Unfortunately, for {\it{any arbitrary drift}} $b$, it is very difficult to find an invariant probability $\phi\d \P$. An easier task is 
to do the ``converse":  Start with a given probability density $\phi$ and find {\it{some $\tilde b$}} so that 
\begin{equation}\label{diffinvdensity}
\frac 12 \Delta \phi= \nabla\cdot (\tilde b\phi),
\end{equation}
i.e., $\phi\d \P$ is an invariant density for the generator $\frac 12 \Delta+ \tilde b\cdot \nabla$ (take for instance, for any $\phi>0$, $\tilde b= \frac{\nabla \phi} \phi$). This is indeed the task set we forth for getting the large deviation lower bound. 
Let $\mathcal E$ be the class of pairs $(\phi, \tilde b)$ so that the probability density $\phi$ satisfies \eqref{diffinvdensity}.
We {\it{tilt}} the original measure $P^{b, \omega}_0$ to $P^{\tilde b, \omega}_0$ for any such $(\phi, \tilde b)\in \mathcal E$. 
The {\it{cost}} for such a tilt is given by the relative entropy 
$$
\begin{aligned}
H\big( P^{\tilde b, \omega}\big| P^{b, \omega} \big)\big|_{\mathcal F_0^t}
&= E^{\tilde b, \omega}\bigg\{ \log\bigg(\frac{\d P^{\tilde b, \omega}}{\d P^{ b, \omega}}\bigg)\bigg|_{\mathcal F_0^t} \bigg\}\\
&=  E^{\tilde b, \omega}\bigg\{\frac 12 \int_0^t \big\| \tilde b\big(\omega(s)\big)- b\big(\omega(s)\big) \big\|^2 \d s\bigg\},
\end{aligned}
$$
by the Cameron- Martin- Girsanov formula (here $\|\cdot\|$ denotes the Euclidean norm in $\R^d$). We remark that for any such pair $(\tilde b,\phi) \in \mathcal E$, $\phi \d \P$ is ergodic 
for the environment process with generator $\frac 12 \Delta+ \tilde b\cdot \nabla$ and hence the ergodic theorem on 
$\Omega$, again by stationarity, translates to the ergodic theorem on $\R^d$, implying 
\begin{equation}\label{differgodic}
\begin{aligned}
&\frac 1t E^{\tilde b, \omega}\bigg\{\frac 12 \int_0^t \big\| \tilde b\big(\omega(s)\big)- b\big(\omega(s)\big) \big\|^2 \d s\bigg\} \longrightarrow\frac 12 \int \|\tilde b(\omega)- b(\omega)\|^2 \phi(\omega) \d \P \quad \P- \mbox{a.s.}
\\
&\mbox{and}\\
&\frac {X_t} t \longrightarrow \int \tilde b(\omega) \phi(\omega) \d \P,
\end{aligned}
\end{equation}
almost surely  with respect to $\P$ and $P^{\tilde b,\omega}$. We fix $\theta\in \R^d$.Then the measure tilting argument, combined with \eqref{differgodic}, leads to the lower bound
\begin{equation}\label{difflb}
\begin{aligned}
\liminf_{t\to\infty} \frac 1t \log E^{b,\omega}\big\{ \e^{\langle \theta, X_t\rangle}\big\} &\geq\sup_{x\in \R^d} \big\{ \langle \theta, x\rangle - I(x)\big\}\\
&= \sup_{(\tilde b, \phi)\in \mathcal E}\bigg\{\int \d \P\phi \bigg(\langle \theta, \tilde b\rangle- \frac 12  \|\tilde b(\omega)- b(\omega)\|^2\bigg)\bigg\}\\
&=: \overline H(\theta),
\end{aligned}
\end{equation}
where
$$
I(x)= \inf_{\heap{(\tilde b, \phi)\in \mathcal E}{\E(\tilde b \phi)=x}} \frac 12 \int \|\tilde b(\omega)- b(\omega)\|^2 \phi(\omega) \d \P.
$$
The desired large deviation principle would follow once we prove the corresponding upper bound 
\begin{equation}\label{diffub}
\begin{aligned}
\limsup_{t\to\infty} \frac 1t \log E^{b,\omega}\big\{ \e^{\langle \theta, X_t\rangle}\big\} &\leq \overline H(\theta),
\end{aligned}
\end{equation}
which contains the heart of the argument. In this context \eqref{diffinvdensity} has an important consequence: For
any suitable test function $g$ on $\Omega$,
$$
\frac 12 \int \d \P\phi \bigg\{ \Delta g + \big\langle \tilde b, \nabla g \big\rangle\bigg\}
\begin{cases}
=0 \, \,\mbox{for all}\,\, g \,\, \mbox{if} \,\, (\tilde b, \phi)\in\mathcal E,
\\
\ne 0\,\, \mbox{for some}\,\, g \,\, \mbox{if} \,\, (\tilde b, \phi)\notin\mathcal E,
\end{cases}
$$
and hence, by taking constant multiples of the left hand side if $(\tilde b, \phi)\notin\mathcal E$,
$$
\inf_g \frac 12 \int \d \P\phi \bigg\{ \Delta g + \big\langle \tilde b, \nabla g \big\rangle\bigg\}=
\begin{cases}
0  \,\, \,\,\,\,\,\mbox{if} \,\, (\tilde b, \phi)\in\mathcal E,
\\
-\infty \,\, \mbox{else}.
\end{cases}
$$
Hence, we can rewrite
$$
\begin{aligned}
\overline H(\theta)
&=\sup_\phi\sup_{\tilde b}\inf_g \bigg\{\int \d \P\phi \bigg(\langle \theta, \tilde b\rangle- \frac 12 \|\tilde b(\omega)- b(\omega)\|^2+\frac 12  \big\{\Delta g + \big\langle \tilde b, \nabla g \big\rangle\big\}\bigg)\bigg\} \\
&\geq \inf_g \sup_\phi\int \d \P \phi \bigg\{ \frac 12 \Delta g+ \langle b, \theta+\nabla g\rangle + \frac 12 \|\theta + \nabla g\|^2\bigg\} \\
&= \inf_g \mathrm{ess}\sup_{\omega- \P}  \bigg\{ \frac 12 \Delta g(\omega)+ \langle b(\omega), \theta+\nabla g(\omega)\rangle + \frac 12 \|\theta + \nabla g(\omega)\|^2\bigg\},
\end{aligned}
$$
where the lower bound above is a result of an approximation argument, subsequent min-max theorems from convex analysis and Lagrange multiplier optimization for $\tilde b$. From this lower bound we would like
to get some ``function" $\tilde g$ on $\Omega$, which is a (sub) -solution to
$$
\frac 12 \Delta \tilde g(\omega)+ \langle b(\omega), \theta+\nabla \tilde g(\omega)\rangle + \frac 12 \|\theta + \nabla \tilde g(\omega)\|^2 \leq \overline H(\theta) \quad \P- \mbox{almost surely}.
$$
It turns out that although $\tilde g$ does not exist as a {\it{function}}, under certain technical conditions (see below), its {\it{gradient}} $\widetilde G= \nabla \tilde g$ exists in $\R^d$ as a function
in some $L^p(\Omega)$ (i.e, has $p$-th moment) and satisfies $\E(\widetilde G)= \theta$ (i.e., has mean $\theta$), $\nabla \times \widetilde G=0$ in the sense that 
$\nabla_i \widetilde G_j= \nabla_j \widetilde G_i$ (i.e, satisfies {\it{closed loop}} condition) and $\P$- almost surely,
$$
\frac 12 \nabla. \widetilde G+ \big\langle b, \tilde G\big\rangle + \frac 12 \big\|\widetilde G\big\|^2 \leq \overline H(\theta),
$$
in the sense of distributions. Via the closed loop (and $p$-th moment) condition, one can define the function ( a {\it{corrector}}) $\widetilde \Psi(\omega, \cdot) \in W^{1,p}_{\mathrm{loc}}(\R^d)$ as
\begin{equation}\label{diffub1}
\widetilde\Psi(\omega,x)= \int_{0 \leadsto x} \langle \widetilde G, \d z(s)\rangle
\end{equation}
where $\widetilde \Psi(0,\omega)=0$ $\P$- almost surely and $z(s)$ is any path connecting $0$ and $x$. Now the large deviation
upper bound \eqref{diffub} follows by a maximum principle argument once we show that $\widetilde \Psi$ has a sub-linear growth
at infinity, i.e., 
\begin{equation}\label{diffsublin}
\|\widetilde\Psi\|_\infty= o(\|x\|) \qquad \mbox{as} \,\|x\|\to\infty
\end{equation}
$\P$- almost surely. This requires substantial technical work and one crucial assumption for this is, existence of $p>d$ and $q>p$ so that
\begin{equation}\label{diffassump}
c_1 \|\nabla u\|^p - c_2 \leq \frac 12 \|\nabla u\|^2+ \langle b, \nabla u\rangle \leq c_3 \|\nabla u\|^q - c_4 
\end{equation}
and the condition $p>d$ allows one to invoke Sobolev imbedding theorem to (locally) control $\|\widetilde \Psi\|_\infty$. Then \eqref{diffsublin} implies the upper bound
\begin{equation}\label{diffHamiltonfor}
 \inf_{\heap{\widetilde G: \nabla\times \widetilde G=0}{\E \widetilde G=\theta}} \mathrm{ess}\sup_{\P} \bigg\{ \frac 12 \|\widetilde G\|^2+ \langle b, \widetilde G\rangle + \frac 12 \nabla\cdot \widetilde G\bigg\} \leq \overline H(\theta)
\end{equation}
Combined with the lower bound \eqref{difflb}, this shows that, in fact equality holds above and proves the existence of the moment generating function (and hence the desired LDP).

We will end this formal discussion on diffusions with random drift by pointing out that the above arguments 
are in fact equivalent to a quenched homogenization effect: Getting the required function $\widetilde G$ as in \eqref{diffub1} is equivalent to getting 
an upper estimate on the solution of 
\begin{equation}\label{homog}
\begin{aligned}
&\partial_t u= \frac 12 \Delta u+ \frac 12 \|\nabla u\|^2+ \langle b, \nabla u\rangle \\
& u(0, x)= \langle \theta, x\rangle,
\end{aligned}
\end{equation}
On the other hand, the Cole-Hopf transform $v= \e^u$ solves
$$
\begin{aligned}
&\partial_t v= \frac 12 \Delta v+ \langle b, \nabla v \rangle \\
&v(0,x)= \e^{\langle \theta, x\rangle},
\end{aligned}
$$
which has a Feynman-Kac representation $v(t,x)= E^{b,\omega}_x \big\{ \e^{\langle \theta, X_t\rangle}\big\}$. Hence, studying the large deviation asymptotics (recall \eqref{difflb} and \eqref{diffub})
$\lim_{t\to\infty}\frac 1t \log v(t,0)$ is same as studying $\lim_{\eps\to 0} \eps \log u(1/\eps,0)=  \lim_{\eps\to 0} u_\eps(1,0)$ where $u_\eps$ solves the rescaled version of \eqref{homog}
$$
\begin{aligned}
&\partial_t u_\eps= \frac 12 \Delta u_\eps+ \frac 12 \|\nabla u_\eps\|^2+ \langle b(x/\eps,\omega), \nabla u\rangle \\
& u_\eps(0, x)= \langle \theta, x\rangle.
\end{aligned}
$$
The limit satisfies 
$$
\begin{aligned}
\partial_t u= \overline H(\nabla u)\\
u(0,x)= \langle \theta, x\rangle,
\end{aligned}
$$
where the {\it{effective Hamiltonian}} $\overline H$ again has the variational formula \eqref{diffHamiltonfor}.




\section{The ergodic theorem and the lower bound.}

In this section we start with the proof of the lower bound asymptotics in Theorem \ref{thmlevel2} and Theorem \ref{thmmomgen}. We need some input from the {\it{environment seen from the particle}}, which, with respect to a
suitably {\it{changed measure}}, possesses important ergodic properties.

\subsection{Markov chains on environments and the ergodic theorem.}
 
Recall that, given the transition probabilities $\pi$ from \eqref{pidef}, for $\P_0$- almost every $\omega\in \Omega_0$, the process $(\tau_{X_n}\omega)_{n\geq 0}$ is a Markov chain with transition kernel
$$
(R_\pi f)(\omega)= \sum_e \pi_\omega(0,e) f(\tau_e \omega),
$$
for every $f$ which is measurable and bounded. 

We need to introduce a class of transition kernels on the space of environments. We denote by $\widetilde \Pi$ the space of functions $\tilde\pi: \Omega_0\times \mathbb U_d\rightarrow [0,1]$  
which are measurable in $\Omega_0$, $\sum_{e\in \mathbb U_d} \tilde\pi(\omega,e)=1$ for almost every $\omega\in \Omega_0$ and for any $\omega\in \Omega_0$ and $e\in \mathbb U_d$, 
\begin{equation}\label{pitilde}
\tilde\pi(\omega,e)=0 \,\,\,\mbox{if and only if}\,\,\, \pi_\omega(0,e)=0.
\end{equation}
For any $\tilde\pi\in\widetilde\Pi$ and $\omega\in \Omega_0$, we define the corresponding quenched probability distribution of the Markov chain $(X_n)_{n\geq 0}$ by
\begin{equation}
\begin{aligned}
&P^{\tilde\pi,\omega}_0(X_0=0)=1\\
&P^{\tilde\pi,\omega}_x(X_{n+1}=x+e| X_n=x)= \tilde\pi(\tau_x\omega,e).
\end{aligned}
\end{equation}

With respect to any $\tilde\pi\in \widetilde \Pi$ we also have a transitional kernel 
$$
(R_{\tilde \pi} f)(\omega)= \sum_{e\in\mathbb U_d} \tilde\pi(\omega,e) f(\tau_e \omega),
$$
for every measurable and bounded $f$. 
For any measurable function $\phi\geq 0$ with $\int \phi \d \P_0=1$, we say that the measure $\phi\d \P_0$ is $R_{\tilde \pi}$-invariant, or simply $\tilde\pi$-invariant, if,
\begin{equation}\label{invdensity}
\phi(\omega)= \sum_{e\in\mathbb U_d} \tilde\pi\big(\tau_{-e}\omega, e\big) \phi\big(\tau_{-e}\omega\big).
\end{equation}
Note that in this case,
\begin{equation*}
\int f(\omega)\phi(\omega)d\P_0(\omega) = \int (R_{\tilde \pi} f) (\omega)\phi(\omega)d\P_0(\omega),
\end{equation*}
for every bounded and measurable $f$.

We denote by $\mathcal E$ such pairs of $(\tilde \pi,\phi)$, i.e.,
\begin{equation}\label{ergodicpair}
\mathcal E= \bigg\{ (\tilde \pi, \phi)\colon \, \tilde\pi\in \widetilde\Pi, \phi\geq 0, \E_0(\phi)=1, \, \phi\d \P_0 \,\mathrm{is}\, \,\tilde\pi-\,\mathrm{invariant}\bigg\}.
\end{equation}
We need an elementary lemma which we will be using frequently. Recall the set $\Mcal_1^\star$ from \eqref{relevantmeasures}.
\begin{lemma}\label{onetoone}
There is a one-to-one correspondence between the sets $\Mcal_1^\star$ and $\mathcal E$.
\end{lemma}
\begin{proof}
Given any $(\tilde\pi, \phi)\in \mathcal E$, we take
\begin{equation}\label{themap}
\begin{aligned}
\d \mu(\omega,e&)= \tilde\pi(\omega,e) \phi(\omega) \,\d \P_0\\
& = \tilde\pi(\omega,e)  \sum_{e: \tau_e\omega^\prime=\omega} \tilde\pi(\omega^\prime,e) \phi(\omega^\prime) \, \d \P_0.
\end{aligned}
\end{equation}
By \eqref{marginals}, $\P_0$-almost surely, 
$$
\begin{aligned}
\d(\mu)_1(\omega) &=\sum_{e\in \mathbb U_d} \d \mu(\omega,e)\\
&= \sum_{e: \tau_e\omega^\prime=\omega} \tilde\pi(\omega^\prime,e) \phi(\omega^\prime) \, \d \P_0\\
&= \sum_{e: \tau_e\omega^\prime=\omega}\d \mu(\omega^\prime,e)\\
&=\d (\mu)_2(\omega). 
\end{aligned}
$$
Hence, $(\mu)_1=(\mu)_2\ll \P_0$. Furthermore, for any edge $e\in \mathbb U_d$ being open in the configuration $\omega$ (i.e., $\omega(e)=1$), 
$$
\frac{\d\mu(\omega,e)}{\d(\mu_1)(\omega)} = \tilde \pi(\omega,e)>0,
$$
recall \eqref{pitilde}. Hence $\mu \in \mathcal M_1^\star$.

Similarly, given any $\mu \in \mathcal M_1^\star$, we can choose
$$
(\tilde \pi, \phi)= \bigg(\frac {\d \mu}{\d (\mu)_1}, \frac {\d (\mu)_1}{\d \P_0}\bigg) \in \mathcal E.
$$
\end{proof}
We also need an ergodic theorem for the environment Markov chain under any transition kernel $\tilde\pi\in\widetilde\Pi$.
This result is standard in the elliptic case. The proof below is an adaptation of the standard proof to our non-elliptic setting.
\begin{theorem}\label{ergodicthm}
Fix $\tilde\pi\in\widetilde\Pi$. If there exists a probability measure $\mathbb Q \ll \P_0$ which is $\tilde\pi$-invariant, then $\mathbb Q\sim \P_0$ and the environment Markov chain with initial law $\mathbb Q$ and transition kernel $\tilde\pi$ is stationary and ergodic for $\P_0$. Moreover, there is at most one probability measure $\mathbb Q$ which is $\tilde\pi$- invariant probability and is absolutely continuous with respect to $\P_0$.
\end{theorem}
\begin{proof}
We fix $\tilde\pi\in \widetilde\Pi$ and let $\mathbb Q\ll \P_0$ be $\tilde \pi$- invariant. We prove the theorem in three steps.

 Let us first show that, $\frac {\d \mathbb Q} {\d \P_0}>0$ $\P_0$- almost surely. This will imply that $\mathbb Q\sim \P_0$. 

Indeed, to the contrary, let us assume that, $0< \P_0(A) <1$ where $A= \big\{\omega\colon \frac {\d \mathbb Q} {\d \P_0}(\omega)>0\big\}$. Then, $\mathbb Q\sim \P_0(\cdot| A)$. 
If we sample $\omega_1 \in \Omega_0$ according to $\mathbb Q$ and $\omega_2$ according to $\tilde\pi(\omega_0,\cdot)$, then 
the distribution of $\omega_2$ is absolutely continuous with respect to $\mathbb Q$ (recall $\mathbb Q$ is $\tilde\pi$ invariant)  and thus, on $A^c$, the distribution of $\omega_2$ has zero measure.

This implies that, for almost every $\omega_1\in A$ and every $e\in \mathbb U_d$ such that $\tilde \pi(\omega_1,e)>0$, $\tau_e \omega_1 \in A$. Since $\tilde \pi \in\widetilde \Pi$, 
for almost every $\omega_1\in A$ and every $e\in \mathbb U_d$ such that $\pi(\omega_1,e)>0$, $\tau_e \omega_1 \in A$.
Now if we sample $\omega_1$ according to $\P_0(\cdot |A)$ and $\omega_2$ according to $\pi(\omega_1,\cdot)$, then, with
probability $1$, $\omega_2\in A$. In other words, $A$ is invariant under $\pi$ (more precisely, $A$ is invariant under the Markov kernel $R_\pi$). Since $\P_0$ is $\pi$-ergodic (see Proposition 3.5 in \cite{BB07}), $\P_0(A)\in \{0,1\}$. By our assumption, $\P_0(A)=1$.

 Now we prove that the environment Markov chain with initial law $\mathbb Q$ and transition kernel $\tilde\pi$ is $\P_0$ ergodic. Let us assume on the contrary, that for some measurable $D$, $\mathbb Q(D)>0$, $\mathbb Q(D^c)>0$ and $D$ is $\tilde \pi$ invariant.
Hence $\P_0(D)>0$ and $\P_0(D^c)>0,$ by the last step. Further, the conditional measure $\mathbb Q_D(\cdot)= \mathbb Q(\cdot| D)$ is $\tilde\pi$ invariant and $\mathbb Q_D\ll \P_0$.
But $\mathbb Q_D(D^c)=0$ and hence, $\frac{\d\mathbb Q_D}{\d \P_0}(D^c)=0$. This contradicts the first step.  

To conclude the proof, we need to prove uniqueness of any $\mathbb Q$ which is $\tilde\pi$- invariant and absolutely continuous with
resect to $\P_0$. Let $\Omega^{\Z}$ be the space of the trajectories $(\dots,\omega_{-1},\omega_0,\omega_1,\dots)$ of the environment 
chain, $\mu_{\mathbb Q}$ the measure associated to the transition kernel $\tilde \pi$ whose finite dimensional distributions are given by
$$
\mu_\mathbb Q\big((\omega_{-n},\dots,\omega_n) \in A\big)= \int_A \mathbb Q(\d \omega_{-n}) \prod_{j=-n}^{n-1}\tilde\pi\big(\omega_j, \d \omega_{j+1}\big).
$$
for any finite dimensional cylinder set $A$ in $\Omega^\Z$. Let $T: \Omega^\Z \longrightarrow \Omega^\Z$ be the shift given by $(T\omega)_n= \omega_{n+1}$ for all $n\in \Z$. 
Since $\mathbb Q$ is $\tilde\pi$- invariant and ergodic, by Birkhoff's theorem,
$$
\lim_{n\to\infty} \sum_{k=0}^{n-1} g \circ T^k = \int g \d \mu_\mathbb Q,
$$
$\mu_\mathbb Q$ (and hence $\mu_{\P_0}$) almost surely for any bounded and measurable $g$ on $\Omega^\Z$. Since the environment chain $(\tau_{X_k}\omega)_{k\geq 0}$ has the same law in $\int P^{\tilde\pi,\omega}_0 \d \mathbb Q$ as $(\omega_0,\omega_1,\dots)$ has in $\mu_\mathbb Q$, if $f(\omega_0)= g(\omega_0,\omega_1,\dots)$, then 
$$
\lim_{n\to\infty} \sum_{k=0}^{n-1} f \circ \tau_{X_k}= \lim_{n\to\infty} \sum_{k=0}^{n-1} g \circ T^k = \int g \d \mu_\mathbb Q= \int f \d \mathbb Q,
$$
for any bounded and measurable $f$ on $\Omega$. The uniqueness of $\mathbb Q$ follows.
\end{proof}

\begin{cor}\label{ergcor}
For any pair $(\tilde\pi,\phi)\in \mathcal E$ and every continuous and bounded function $f: \Omega_0 \times  \mathbb U_d \rightarrow \R$, 
$$
\lim_{n\to\infty} \frac 1n \sum_{k=0}^{n-1} f(\tau_{X_k}\omega, X_{k+1}-X_k)= \int_{\Omega_0} \, \d\P_0\, \phi(\omega) \sum_e f(\omega,e) \tilde\pi(\omega,e) \quad\P_0- \, \mbox{almost surely}.
$$
\end{cor}
\begin{proof}
This is an immediate consequence of Theorem \ref{ergodicthm} and Birkhoff's ergodic theorem.
\end{proof}
\subsection{The lower bounds of Theorem \ref{thmlevel2} and Theorem \ref{thmmomgen}.}

We are ready to deduce the lower bound \eqref{ldplb}. Recall the definition of $\mathfrak I$ from \eqref{Idef}.

\begin{lemma}\label{lemmalb}
For any open set $\mathcal G$ in $\Mcal_1(\Omega_0 \times  \mathbb U_d)$, $\P_0$- almost surely,
\begin{equation}\label{eqlemmalb}
\begin{aligned}
\liminf_{n\to\infty} \frac 1n \log \P^{\pi,\omega}_0 \big(\mathcal L_n \in \mathcal G\big) &\geq - \inf_{\mu\in \mathcal G} \mathfrak I(\mu) \\
&=- \inf_{\mu\in \mathcal G} \mathfrak I^{\star\star}(\mu).
\end{aligned}
\end{equation}
\end{lemma}
\begin{proof}
For the lower bound in \eqref{eqlemmalb}, it is enough to show that, for any $\mu\in \Mcal_1^\star$ and any open neighborhood $\mathcal U$ containing $\mu$,
\begin{equation}\label{lb0}
\liminf_{n\to\infty} \frac 1n \log P^{\pi,\omega}_0\big(\mathcal L_n \in \mathcal U\big) \geq - \mathfrak I(\mu).
\end{equation}
Given $\mu\in\Mcal_1^\star$, from Lemma \ref{onetoone} we can get the pair 
\begin{equation}\label{map}
(\tilde \pi, \phi)= \bigg(\frac {\d \mu}{\d (\mu)_1}, \frac {\d (\mu)_1}{\d \P_0}\bigg) \in \mathcal E,
\end{equation}
and by Theorem \ref{ergodicthm}, 
\begin{equation}\label{lb1}
\lim_{n\to\infty} P^{\tilde\pi,\omega}_0\big(\mathcal L_n \in \mathcal U\big) =1.
\end{equation}
Further,
$$
\begin{aligned}
P^{\pi,\omega}_0\big(\mathcal L_n \in \mathcal U\big) &= E^{\tilde\pi,\omega}_0 \bigg\{\1_{\{\mathcal L_n \in \mathcal U\}} \frac{\d P_0^{\pi,\omega}}{\d P_0^{\tilde\pi,\omega}}\bigg\} \\
&=\int \d P^{\tilde\pi,\omega}_0 \bigg\{\1_{\{\mathcal L_n \in \mathcal U\}} \exp\bigg\{-\log\,\frac{\d P_0^{\tilde\pi,\omega}}{\d P_0^{\pi,\omega}}\bigg\}\bigg\} .
\end{aligned}
$$
Hence, by Jensen's inequality, 
$$
\begin{aligned}
\liminf_{n\to\infty}\frac 1n \log P^{\pi,\omega}_0\big(\mathcal L_n \in \mathcal U\big) 
&\geq \liminf_{n\to\infty}\frac 1n \log P^{\tilde\pi,\omega}_0\big(\mathcal L_n \in \mathcal U\big) \\
&\qquad-\limsup_{n\to\infty}\frac 1{nP^{\tilde\pi,\omega}_0\big(\mathcal L_n \in \mathcal U\big)} \int_{\{\mathcal L_n \in \mathcal U\}} \d P^{\tilde\pi,\omega}_0 \bigg\{\log\,\frac{\d P_0^{\tilde\pi,\omega}}{\d P_0^{\pi,\omega}}\bigg\}\\
&= -\int \d\P_0(\omega)\,\phi(\omega) \sum_{|e|=1}  \tilde\pi(\omega,e)  \log \frac{\tilde\pi(\omega,e)}{\pi_\omega(0,e)}\\
&=- \mathfrak I(\mu),
\end{aligned}
$$
where the first equality follows from \eqref{lb1} and corollary \ref{ergcor} and the second equality follows from \eqref{map}. This proves \eqref{lb0}. Finally, since 
$\mathcal G$ is open, $\inf_{\mu\in \mathcal G} \mathfrak I(\mu)= \inf_{\mu\in \mathcal G} \mathfrak I^{\star\star}(\mu)$ (see \cite{R70}). This proves the equality in \eqref{eqlemmalb} and the lemma.
\end{proof}
\begin{cor}\label{corlb}
For every continuous and bounded function $f: \Omega_0 \times  \mathbb U_d\longrightarrow \R$ and for $\P_0$-almost every $\omega\in \Omega_0$, 
\begin{equation}\label{lb}
\begin{aligned}
\liminf_{n\to\infty} \frac 1n\log E^{\pi,\omega}_0 \bigg\{ \exp \bigg(\sum_{k=0}^{n-1} f\big(\tau_{X_k}\omega, X_{k+1}-X_k\big)\bigg)\bigg\} &\geq \sup_{\mu\in \Mcal_{1}^{\star}} \big\{\langle f,\mu\rangle- \mathfrak I(\mu)\big\}\\
&=\sup_{\mu\in \Mcal_{1}(\Omega_0 \times  \mathbb U_d)} \big\{\langle f,\mu\rangle- \mathfrak I(\mu)\big\}.
\end{aligned}
\end{equation}
\end{cor}
\begin{proof}
This follows immediately from Varadhan's lemma and Lemma \ref{lemmalb}.
\end{proof}

\section{A class $\mathcal G_\infty$ of gradients.}

We introduce a class of functions which will play an important role for the large deviation analysis to follow. We say that a function $G: \Omega_0 \times  \mathbb U_d \longrightarrow \R$ is in class $\mathcal G_\infty$ if it satisfies the 
following conditions:
\begin{itemize}
\item {\bf{Zero mean:}} For every $e\in \mathbb U_d$, 
\begin{equation}\label{meanzero}
\E_0\big(G(\cdot,e)\big)=0.
\end{equation}
\item {\bf{Uniform boundedness.}} For every $e\in \mathbb U_d$, 
\begin{equation}\label{unifbound}
\mathrm{ess} \sup_{\omega-\P_0} G(\omega,e) =M< \infty.
\end{equation}
\item {\bf{Closed loop.}} Let $(x_0,\dots,x_n)$ be a closed loop on the infinite cluster $\mathcal C_\infty$ (i.e., $x_0,x_1,\dots, x_n$ is 
 a nearest neighbor occupied path so that $x_0=x_n$). Then,
 \begin{equation}\label{closedloop}
 \sum_{j=0}^{n-1} G(\tau _{x_j} \omega, x_{j+1}-x_j) =0 \quad \P_0- \mbox{almost surely}.
 \end{equation}
  \end{itemize}
For any $G\in \mathcal G_\infty$, the closed loop condition has two important consequences in the present context: First, along any nearest neighbor occupied path $(x_0, x_1, \dots, x_n)$ so that $x_0=0$ and $x_n=x$ on $\mathcal C_\infty$, we can define the function 
  \begin{equation}\label{Psidef}
 \Psi_G(\omega,x)=  \Psi(\omega,x)=  \sum_{j=0}^{n-1} G(\tau _{x_j} \omega, x_{j+1}-x_j) 
  \end{equation}
by the closed loop condition, this definition is independent of the chosen path for almost every $\omega\in \{x\in \mathcal C_\infty\}$.

Secondly, again by the closed loop condition, $\Psi$ satisfies
\begin{itemize}
\item {\bf{Shift covariance:}} For $\P_0$-almost every $\omega\in \Omega_0$ and all $x, y\in \mathcal C_\infty$,
$$
\Psi(\omega,x)- \Psi(\omega,y)= \Psi(\tau_y\omega, x-y).
$$
\end{itemize}

For any given $G\in \mathcal G_\infty$, let us fix $\Psi$, which satisfies an important property.
  \begin{theorem}[Sub-linear growth at infinity on the cluster]\label{sublinearthm}
 For any $G\in \mathcal G_\infty$, $\Psi$ has at most sub-linear growth at infinity on the infinite cluster $\P_0$- almost surely,. In other words,
  $$
  \lim_{n\to\infty} \max_{\heap{x\in \mathcal C_\infty}{|x|\leq n}}  \frac {|\Psi(\omega,x)|}n =0.
  $$
  \end{theorem}
  Before we present the proof let us collect some useful facts which will finish the proof of the theorem. First, we start with a weaker version of the above result.
  \begin{lemma}[Sub-linearity on average]\label{averagesublinear}
  For every $\eps>0$, 
  \begin{equation}
  \lim_{n\to\infty} \frac 1{n^d} \sum_{\heap {x\in \mathcal C_\infty}{|x|\leq n}} \, \1_{\big\{|\Psi(x,\omega|>\eps n\big\}} =0 \qquad \P_0-\,\mbox{almost surely.}
  \end{equation}
  \end{lemma}
  \begin{proof}
  This follows verbatim the proof of Theorem 5.4 in \cite{BB07} (Interestingly, in \cite{BB07}, along with the mean zero and shift-covariance, $\Psi$ being only square integrable with respect to $\P_0$
  is enough to deduce the above result).
  \end{proof}
  
We also need the following version of the classical result of Antal-Pisztora (\cite{AP96}) 
  about the {\it{chemical distance}} of two points in the cluster. Indeed, for $p>p_c(d)$ and $x,y \in \mathcal C_\infty$, let $\d_{\mathrm ch}(x,y)$ denote the minimal length of an open
  path connecting $x$ and $y$. 
  \begin{lemma}\label{chemdist}
  Fix $\delta>0$. Then there exists a constant $c=c(p,d)$ such that, $\P_0$- almost surely, for every $n$ large enough and points $x,y\in \mathcal C_\infty$ with $|x|<n, |y|< n$ and $|x-y|< \delta n$, we have
  $\d_{\mbox{ch}}(x,y) < c \delta n$.
  \end{lemma}
  \begin{proof}
  The statement of this lemma, which is slightly stronger than what is stated in Antal-Pisztora (\cite{AP96}) follows from Lemma 2.14 in \cite{DGK01}.
  \end{proof}
  We need another elementary fact. Let $\theta(p)$ denote the percolation density, i.e., $\theta(p)$ is the probability that $0$ is in the infinite open cluster.
  \begin{lemma}\label{Birkhoff}
  Fix $\delta>0$. For every $n$ large enough, in a ball of radius $\delta n$ in $[-n,n]^d$ there are at least 
  $\delta^d (2n)^d \,\frac {\theta(p)} 2$ points in $\mathcal C_\infty$. 
  \end{lemma}
  \begin{proof}
  This is an easy consequence of Birkhoff's ergodic theorem.
  \end{proof}
  We are now ready to prove Theorem \ref{sublinearthm}. 
  
{\it{Proof of Theorem \ref{sublinearthm}:}}
Let $\eps>0$ be arbitrary, $\delta= \frac 12 \big(\frac{4\eps}{\theta(p)}\big)^{\frac 1d}$ and $n$ large enough so that the following three implications hold:
\begin{itemize}
\item
\begin{equation}\label{eq1}
\sum_{\heap {x\in \mathcal C_\infty}{|x|\leq n}} \, \1_{\big\{|\Psi(x,\omega|>\eps n\big\}} < \eps n^d.
\end{equation}
\item For any $x,y\in \mathcal C_\infty$ with $|x|<n$, $|y|<n$ and $|x-y| < \delta n$, 
\begin{equation}\label{eq2}
\d_{\mbox{ch}}(x,y) < c \delta n
\end{equation}
\item 
\begin{equation}\label{eq3}
\begin{aligned}
\#\big\{\mbox{points in a box of radius}\,\, \delta n\,\,\mbox{in}\,\,[-n,n]^d\,\,\mbox{in}\,\, \mathcal C_\infty\} 
> 2\eps n^d.
\end{aligned}
\end{equation}
\end{itemize}
These are consequences of Lemma \ref{averagesublinear}, Lemma \ref{chemdist} and Lemma \ref{Birkhoff} respectively. We note that, for such small $\eps>0$ and large $n$ and every $x\in [-n,n]^d$, there exists $y\in  [-n,n]^d \cap \mathcal C_\infty$ so that $|y-x|< \delta n$ and $|\Psi(\omega,x)| \leq \eps n$, $\P_0$-almost surely. Indeed, by \eqref{eq1} there are at most $\eps n^d$ points $z\in [-n,n]^d$ such that $|\Psi(\omega,z)| \geq \eps n$ and by \eqref{eq2}, there are at least $2\eps n^d$ points in $B_{n\delta}(x)\cap \mathcal C_\infty$. Hence, we have at least one point $y\in  [-n,n]^d \cap \mathcal C_\infty$ such that   $|y-x|< \delta n$ and $|\Psi(\omega,y)| \leq \eps n$, 
$\P_0$- almost surely.

Recall the definition of $\Psi$ from \eqref{Psidef}. Then, by \eqref{eq2},
$$
\begin{aligned}
\big| \Psi(\omega,x) - \Psi(\omega,y)\big| &\leq \d_{\mbox{ch}}(x,y) \,\mbox{ess} \sup_{\omega- \P_0} G(\omega, x)\\
&\leq c \delta n M,
\end{aligned}
$$
 for some $M< \infty$, recall \eqref{unifbound}. Hence, $\P_0$- almost surely, 
$$
\begin{aligned}
|\Psi(\omega,x)| &\leq |\Psi(\omega,y)|+ c \delta n M \\
& \leq\eps n+ c \delta n M.
\end{aligned}
$$
Since $\eps>0$ is arbitrary, Theorem \ref{sublinearthm} is proved.
\qed

We end with a corollary to Theorem \ref{sublinearthm}.
\begin{cor}
Let $G\in \mathcal G_\infty$. For every $\eps>0$,  there exists $c_\eps=c_\eps(\omega)$ so that, for every sequence of points $(x_k)_{k=0}^n$ on $\mathcal C_\infty$ with $x_0=0$ and $|x_{k+1}-x_{k}|=1$,
$$
\bigg|  \sum_{k=0}^{n-1} G(\tau_{x_k} \omega, x_{k+1}-x_k)\bigg| \leq c_\eps+n\eps.
$$
In particular, 
\begin{equation}\label{ub2}
 \sum_{k=0}^{n-1} G(\tau_{x_k} \omega, x_{k+1}-x_k)\geq -c_\eps- n\eps.
\end{equation}
\end{cor}
\section{Limiting Logrithmic moment generating functions: proof of Theorem \ref{thmmomgen}.}
In view of Corollary \ref{corlb}, Theorem \ref{thmmomgen} will be proved as soon as prove an upper bound
of the limiting logarithmic moment generating function which matches the right hand side of \eqref{lb}. We first prove an upper bound
based on the sub-linear growth property of gradient functions from the last section and subsequently show that this upper bound matches the lower bound in \eqref{lb}.
\subsection{The upper bound.}
\begin{lemma}\label{ub}
For $\P_0$- almost every $\omega\in \Omega_0$,  
$$
\limsup_{n\to\infty} \frac 1n\log \E_{0,\omega} \bigg\{ \exp \big\{\sum_{k=0}^{n-1} f\big(\tau_{X_k}\omega, X_{k+1}-X_k\big)\big\}\bigg\} \leq \inf_{G\in \mathcal G_\infty}\Lambda(f, G),
$$
where
\begin{equation}\label{UfG}
\Lambda(f, G)=\mathrm{ess}\sup_{\omega\in \Omega_0} \log \sum_{e} \1_{\{\omega(e)=1\}}\pi_\omega(0,e) \exp\big\{f(\omega,e)+ G(\omega,e)\big\},
\end{equation}
\end{lemma}

\begin{proof}
Fix $G\in \mathcal G_\infty$. By the definition of the Markov chain $P^{\pi,\omega}_0$ we have, 
$$
\begin{aligned}
&E^{\pi,\omega}_0\bigg\{\exp\bigg\{ f(\tau_{X_{k}}\omega,X_{k+1}-X_{k})+G(\tau_{X_{k+1}}\omega,X_{k+1}-X_{k})\bigg\}\bigg|X_{k}\bigg\} \\
&= \sum_{|e|=1} \pi_\omega\big(X_{k},X_{k}+e\big) \e^{f(\tau_{X_{k}}\omega,e)+ G(\tau_{X_{k}}\omega,e)}\\
&=\sum_{|e|=1}\1_{\{(\tau_{X_{k}}\omega)(e)=1\}} \pi_\omega\big(X_{k},X_{k}+e\big) \e^{f(\tau_{X_{k}}\omega,e)+ G(\tau_{X_{k}}\omega,e)}\\
&\leq \e^{\Lambda(f, G)},
\end{aligned}
$$
where the uniform upper bound follows from \eqref{UfG}.

Invoking the Markov property and successive conditioning, we have

\begin{equation}\label{ub1}
E^{\pi,\omega}_0\bigg\{\exp\bigg\{\sum_{k=0}^{n-1} \bigg(f(\tau_{X_k}\omega,X_{k+1}-X_{k})+G(\tau_{X_k}\omega,X_{k+1}-X_{k})\bigg)\bigg\}\bigg\} \leq \e^{n\Lambda(f, G)}.
\end{equation}
Plugging the lower bound \eqref{ub2} in \eqref{ub1}, dividing by $n$ on both sides, taking logarithm and passing to $\limsup_{n\to\infty}$ we have the upper bound
$$
\limsup_{n\to\infty} \frac 1n\log \E_{0,\omega} \bigg\{ \exp \big\{\sum_{k=0}^{n-1} f\big(\tau_{X_k}\omega, X_{k+1}-X_k\big)\big\}\bigg\} \leq \Lambda(f, G)+ \eps.
$$
Passing to $\eps\to 0$ and subsequently taking $\inf_{G\in \mathcal G_\infty}$ we finish the proof of the lemma.
\end{proof}


  
\subsection{Equivalence of bounds: Variational analysis.}
  
We pick up from the lower bound \eqref{lb} and denote this variational formula by
\begin{equation}\label{Lf}
\begin{aligned}
\overline H(f)&=\sup_{\mu\in \Mcal_1^\star} \big\{\langle f, \mu\rangle- \mathfrak I(\mu)\big\}\\
&= \sup_{(\tilde\pi,\phi)\in \mathcal E} \bigg\{\int \d\P_0(\omega)\,\phi(\omega) \sum_{|e|=1}  \tilde\pi(\omega,e) \bigg\{f(\omega,e) - \log \frac{\tilde\pi(\omega,e)}{\pi_\omega(0,e)}\bigg\}\bigg\},
\end{aligned}
\end{equation}
recall from Lemma \ref{onetoone} the one-to-one correspondence between elements of the set $\Mcal_1^\star$ and the pairs $\mathcal E$ (and \eqref{map}, \eqref{Idef}). In
this section we show that $\overline H(f)$ equals the upper bound obtained in the last subsection. Modulo some care about containment in the infinite cluster, 
the line of arguments follow parallel to \cite{KRV06} (and also \cite{R06}, {\cite{Y08}). 

 \begin{prop}\label{propequal}
 For every $\eps>0$, there is some $G_\eps\in \mathcal G_\infty$ so that
\begin{equation}\label{equiv0}
 \Lambda(f,G_\eps) \leq\overline H(f)
+ \eps,
\end{equation}
where $\Lambda(f, G_\eps)$ is defined in \eqref{UfG}. Hence, by Corollary \ref{corlb} and Lemma \ref{ub}, 
$$\overline H(f)= \inf_{G\in \mathcal G_\infty} \Lambda(f, G).
$$
\end{prop}
\begin{proof}
We split the proof into several steps.

{\bf{Step 1: Convex analysis.}}

We recall the definition of pairs $\mathcal E$ from \eqref{invdensity} and note that
$\overline H(f)$ can be rewritten as 
\begin{equation}\label{equiv1}
\sup_{\phi}\sup_{\tilde\pi\in \widetilde\Pi}\inf_{g}\bigg[\int \d\P_0(\omega)\,\phi(\omega) \bigg\{\sum_{e\in \mathbb U_d}  \tilde\pi(\omega,e) \bigg(f(\omega,e) - \log \frac{\tilde\pi(\omega,e)}{\pi_\omega(0,e)}\bigg)-\big\{g(\tau_e\omega)-g(\omega)\big\}\bigg\}\bigg]
\end{equation}
since, by \eqref{invdensity},
$$
\inf_g\int \d\P_0\phi(\omega) \sum_{e\in \mathbb U_d} \tilde\pi(\omega,e)\bigg\{g(\tau_e\omega)-g(\omega)\bigg\}=
\begin{cases}
0\quad \mathrm{if}\,(\phi,\tilde\pi)\in \mathcal E \\
-\infty \quad \mbox{else,}
\end{cases}
$$
with the infimum over $g$ being taken over all bounded $\mathcal B$- measurable real functions. 

In order to proceed with the variational analysis, we would like to swap the order of $\sup_{\tilde\pi\widetilde\Pi}$ and $\inf_g$ in \eqref{equiv1} by invoking a {\it{min-max}} theorem which requires a compactness argument.
We choose a sequence of finite $\sigma$-algebras $(\mathcal B_k)_{k\geq 1}$ so that, for each $k$, $\mathcal B_k$ contains information about
all open (and hence closed) bonds in the box of size $k$ around the origin. Then, for each $k\geq 1$, $\pi_{\cdot}(0,e)$ is measurable with respect to $\mathcal B_k$, for each $e\in \mathbb U_d$. Furthermore, $\mathcal B_k\subset\tau_e \mathcal B_{k+1}$ for all $e\in \mathbb U_d$ and $\mathcal B= \sigma(\cup_{k\geq 1} \mathcal B_k)$, where $\mathcal B$ is the Borel $\sigma$- algebra.

For each $k\geq 1$ we can restrict both the supremums in \eqref{equiv1} to $\mathcal B_k$-measurabale probability densities $\phi$ and $\mathcal B_k$-measurable $\tilde\pi\in \widetilde\Pi$ and get a further lower bound 
$$
\begin{aligned}
 \overline H(f) \geq &\sup_{\phi}\sup_{\tilde\pi\in \widetilde\Pi}\inf_{g}\bigg[\int \d\P_0(\omega)\,\phi(\omega) \\
 &\qquad\qquad\qquad\bigg\{\sum_{e\in \mathbb U_d}  \tilde\pi(\omega,e) \bigg(f(\omega,e) - \log \frac{\tilde\pi(\omega,e)}{\pi_\omega(0,e)}\bigg)-\big\{g(\tau_e\omega)-g(\omega)\big\}\bigg\}\bigg],
 \end{aligned}
 $$
Since each $\mathcal B_k$ is finite, the supremum over $\tilde\pi$ is taken over a compact set. Further, since the integral above is concave and continuous in $\tilde\pi$ and linear (in particular, convex) in $g$, a min-max argument (\cite{F53}) 
allows us to change the order of $\sup_{\tilde\pi\in \widetilde\Pi}$ and $\inf_{g}$, leading to  
\begin{equation}\label{equiv2}
\begin{aligned}
\overline H(f) \geq &\sup_{\phi}\inf_{g}\sup_{\tilde\pi\in \widetilde\Pi}\bigg[\int \d\P_0(\omega)\,\phi(\omega) \\
 &\qquad\qquad\qquad\bigg\{\sum_{e\in \mathbb U_d}  \tilde\pi(\omega,e) \bigg(f(\omega,e) - \log \frac{\tilde\pi(\omega,e)}{\pi_\omega(0,e)}\bigg)-\big\{g(\tau_e\omega)-g(\omega)\big\}\bigg\}\bigg].
 \end{aligned}
 \end{equation}
 We can take conditional expectation of the integrand above with respect to $\mathcal B_k$ and use that both $\phi$ and $\tilde\pi$ are $\mathcal B_k$-measurable to rewrite the right hand side above as
 \begin{equation}\label{equiv2mid}
 \begin{aligned}
&  \sup_{\phi}\inf_{g}\sup_{\tilde\pi\in \widetilde\Pi}\bigg[\int \d\P_0(\omega)\,\phi(\omega) 
  \bigg\{\sum_{e\in \mathbb U_d}  \tilde\pi(\omega,e) \big\{h(\omega,e)-\log \tilde\pi(\omega,e)\big\}\bigg\}\bigg] ,
  \end{aligned}
\end{equation}
where
 $$
 \begin{aligned}
h(\omega,e)&=\E_0\big\{\log \pi_\omega(0,e) +f(\omega,e)+g(\omega)-g(\tau_e\omega)\big|\mathcal B_k\big\} \\
&=\log \pi_\omega(0,e) +\E_0\big\{f(\omega,e)+g(\omega)-g(\tau_e\omega)\big|\mathcal B_k\big\},
\end{aligned}
$$  
and we used that by our choice, $\pi$ is also $\mathcal B_k$- measurable. Staring at \eqref{equiv2mid} we note the local dependence of the integrand on $\tilde\pi$ allowing us to bring the supremum over $\tilde\pi$ inside the integral, leading to 
$$
\overline H(f)\geq 
\sup_{\phi}\inf_{g}\bigg[\int \d\P_0(\omega)\,\phi(\omega) 
  \sup_{\tilde\pi\in \widetilde\Pi} \bigg\{\sum_{e\in \mathbb U_d}  \tilde\pi(\omega,e) \big\{h(\omega,e)-\log \tilde\pi(\omega,e)\big\}\bigg\}\bigg].
$$

A direct Lagrange multiplier computation shows that the supremum in the integrand above is attained at 
$$
\begin{aligned}
\tilde\pi(\omega,e)&= \frac{\e^{h(\omega,e)}}{\sum_{e\in \mathbb U_d}  \e^{h(\omega,e)}} 
\\
&=\frac{\pi_\omega(0,e) \,\,\exp\big\{\E_0\big(f(\omega,e)+g(\omega)-g(\tau_e\omega)\big|\mathcal B_k\big)\big\}}
{\sum_{|e|=1}  \pi_\omega(0,e) \,\,\exp\big\{\E_0\big(f(\omega,e)+g(\omega)-g(\tau_e\omega)\big|\mathcal B_k\big)\big\}} \in \widetilde\Pi(\omega,e).
\end{aligned}
$$
Replacing this value in the integrand leads to the lower bound
\begin{equation}\label{equiv3}
\begin{aligned}
\overline H(f) &\geq \sup_{\phi}\inf_{g}\bigg[\int \d\P_0(\omega)\,\phi(\omega)  \\
&\qquad\qquad\qquad \log \sum_{e\in \mathbb U_d} \pi_\omega(0,e) \e^{\E_0\big\{f(\omega,e)+ g(\omega)-g(\tau_e\omega)| \mathcal B_k\big\}}\bigg]
\end{aligned}
\end{equation}
Again using a similar min-max argument to change to the order of supremum and the infimum and subsequently replacing the supremum over $\phi$ with the integral $\int \d \P_0 \phi$ by $\mathrm{ess}\sup_{\omega-\P_0}$ we arrive at the lower bound 
$$
\overline H(f) \geq \inf_g \mathrm{ess}\sup_{\omega-\P_0}\bigg\{ \log \sum_{e\in \mathbb U_d} \pi_\omega(0,e) \e^{\E_0\big\{f(\omega,e)+ g(\omega)-g(\tau_e\omega)| \mathcal B_k\big\}}\bigg\} .
$$
Note that, we can restrict the above lower bound to
\begin{equation}\label{equiv4}
\overline H(f) \geq \inf_g \mathrm{ess}\sup_{\omega-\P_0}\bigg\{ \log \sum_{e\in \mathbb U_d} \, \1_{\{\omega(e)=1\}} \pi_\omega(0,e) \e^{\E_0\big\{f(\omega,e)+ g(\omega)-g(\tau_e\omega)| \mathcal B_k\big\}}\bigg\} .
\end{equation}

\medskip\noindent

{\bf{Step 2: Approximate gradient and uniform boundedness.}}

This implies, that, for every $\eps>0$ and $k\geq 1$, there exists a bounded $\mathcal B$ measurable function $g_{k,\eps}$ so that, for $\P_0$- almost every $\omega \in \Omega_0$, and $e\in \mathbb U_d$,
\begin{equation} \label{equiv5}
 \1_{\{\omega(e)=1\}} \E_0\big\{ g_{k,\eps}(\omega)-g_{k,\eps}(\tau_e\omega)| \mathcal B_k\big\} \leq -\1_{\{\omega(e)=1\}} \log \pi_\omega(0,e) + \|f\|_\infty+ \overline H(f) + \eps.
 \end{equation}
 We set, 
 \begin{equation}\label{gradcorrect}
 G_{k,\eps}(\omega,e)=  \1_{\{\omega(e)=1\}} \E_0\big\{ g_{k,\eps}(\omega)-g_{k,\eps}(\tau_e\omega)| \mathcal B_{k-1}\big\}.
 \end{equation}
We would like to show that, for every $\eps>0$, the family $\{G_{k,\eps}(\cdot,e)\}_{k\geq 1}$ is uniformly bounded in the essential supremum norm.  

First, by taking conditional expectation on both sides of \eqref{equiv5} with respect to $\mathcal B_{k-1}$ we have
 \begin{equation}\label{equiv6}
 \begin{aligned}
 G_{k,\eps}(\omega,e) &\leq -\E_0\big\{\1_{\{\omega(e)=1\}} \log \pi_\omega(0,e)| \mathcal B_{k-1}\big\} + \|f\|_\infty+ \Gamma(f) + \eps  \\
 &=-\1_{\{\omega(e)=1\}} \log \pi_\omega(0,e) + \|f\|_\infty+ \overline H(f) + \eps\\
 &\leq \log(2d)+ \|f\|_\infty+ \overline H(f) + \eps,
  \end{aligned}
  \end{equation}
since the random walk transition probabilities are bounded away from zero on the event $\{\omega(e)=1\}$. We can also reverse the argument to get a lower bound: Indeed, if $H(\omega,e)=H_{k,\eps}(\omega,e)=g_{k,\eps}(\omega)-g_{k,\eps}(\tau_e\omega)$, then $H(\omega,-e)=- H(\tau_{-e}\omega,e)$.
Now applying \eqref{equiv5} again for the edge $-e$, we have
\begin{equation}\label{equiv7}
\begin{aligned}
\Gamma(f)+ \|f\|_\infty - \1_{\{\omega(-e)=1\}} \log \pi_\omega(0,-e) + \eps &\geq \E_0 \big\{H(\omega,-e)| \mathcal B_k\big\} \1_{\{\omega(-e)=1\}} \\
&= \E_0 \big\{-\1_{\{(\tau_{-e}\omega)(e)=1\}}H(\tau_{-e}\omega,e)| \mathcal B_k\big\} \\
&{=} \E_0 \big\{-\1_{\{\omega(e)=1\}}H(\omega,e)| \tau_e\mathcal B_k\big\},
\end{aligned}
\end{equation}
where the second equality follows from the symmetry of the conductances. Again taking conditional expectation on both sides with respect to $\mathcal B_{k-1}$ and recalling that $\mathcal B_{k-1}\subset \tau_e \mathcal B_k$, we have
\begin{equation}\label{equiv8}
\begin{aligned}
-G_{k,\eps}(\omega,e) &\leq -\E_0\big\{\1_{\{\omega(-e)=1\}} \log \pi_\omega(0,-e)| \mathcal B_{k-1}\big\} + \|f\|_\infty+ \overline H(f) + \eps \\
 &=-\1_{\{\omega(-e)=1\}} \log \pi_\omega(0,-e) + \|f\|_\infty+ \overline H(f) + \eps\\
 &\leq \log(2d)+\|f\|_\infty+ \overline H(f) + \eps,
  \end{aligned}
 \end{equation}
 since the transition probabilities are again bounded away from zero on the event $\{\omega(-e)=1\}$. We combine \eqref{equiv6} and \eqref{equiv8} to conclude that, for some non-random constant $M<\infty$,
 \begin{equation}\label{unifbd}
\mathrm{ess}\sup_{\omega-\P_0}G_{k,\eps}(\omega,e) \leq M \quad\forall k\geq 1.
\end{equation}

 {\bf{Step 3: Convergence to the gradient $G_\eps\in \mathcal G_\infty$.}}
 
 Since $\{G_{k,\eps}(\cdot,e)\}_{k\geq 1}$ is a uniformly bounded family, it is weakly compact and weakly converges, possibly along some subsequence, 
 to some $G_\eps(\cdot,e)$, which clearly has $\E_0$ expectation zero and is also uniformly bounded in the essential supremum norm. 
 
 Let $(x_0,\dots,x_n)$ be a closed loop on $\mathcal C_\infty$ (i.e., $x_0,x_1,\dots, x_n$ is 
 a nearest neighbor occupied path so that $x_0=x_n$). Note that, weak convergence preserves conditional expectation. Then, for any fixed $l\geq 1$,
 $$
 \begin{aligned}
&\mathbb E_0 \bigg\{\sum_{j=0}^{n-1} G_\eps(\tau_{x_j}\omega, x_{j+1}-x_j)\bigg| B_l\bigg\}\\
&=\mathbb E_0 \bigg\{\sum_{j=0}^{n-1} \lim_{k\to\infty}G_{k,\eps}(\tau_{x_j}\omega, x_{j+1}-x_j)\bigg| B_l\bigg\}\\
&=\sum_{j=0}^{n-1} \lim_{k\to\infty} \E_0\bigg\{\1_{\big\{(\tau_{x_j}\omega)(x_{j+1}-x_j)=1\big\}} \E_0\bigg(g_{k,\eps}(\omega)-g_{k,\eps}(\tau_{x_{j+1}-x_j}\omega)\big| \mathcal B_{k-1}\bigg)\big(\tau_{x_j}(\omega)\big)\bigg| B_l\bigg\} \\
&=\sum_{j=0}^{n-1} \lim_{k\to\infty} \E_0\bigg\{\1_{\big\{\omega(x_{j+1}-x_j)=1\big\}} \E_0\bigg(g_{k,\eps}(\tau_{x_j}\omega)-g_{k,\eps}(\tau_{x_{j+1}}\omega)\big| \tau_{-x_j}\mathcal B_{k-1}\bigg)\bigg| B_l\bigg\}\\
&=\sum_{j=0}^{n-1} \lim_{k\to\infty} \E_0\bigg\{ \E_0\bigg(g_{k,\eps}(\tau_{x_j}\omega)-g_{k,\eps}(\tau_{x_{j+1}}\omega)\big| \tau_{-x_j}\mathcal B_{k-1}\bigg)\bigg| B_l\bigg\}.
\end{aligned}
$$
For $k$ large enough, $B_l\subset \tau_{-x_j}\mathcal B_{k-1} $ and hence, by the tower property the last term equals
$$
\begin{aligned}
&\sum_{j=0}^{n-1} \lim_{k\to\infty} \E_0\bigg(g_{k,\eps}(\tau_{x_j}\omega)-g_{k,\eps}(\tau_{x_{j+1}}\omega)\bigg| B_l\bigg)\\
&= \lim_{k\to\infty}\sum_{j=0}^{n-1} \E_0\bigg(g_{k,\eps}(\tau_{x_j}\omega)-g_{k,\eps}(\tau_{x_{j+1}}\omega)\bigg| B_l\bigg)\\
&=0.
\end{aligned}
$$
This implies that, 
$$
\sum_{j=0}^{n-1} G_\eps(\tau_{x_j}\omega, x_{j+1}-x_j) =0 \qquad \P_0-\,\mbox{almost surely.}
$$
We conclude that $G_\eps\in \mathcal G_\infty$.

\medskip\noindent

{\bf{Step 4: Conclusion.}}

Recall the lower bound \eqref{equiv4}. Since $\E_0(f(\omega,e)|\mathcal B_{k-1})$ is a bounded martingale, it converges almost surely to $f(\cdot,e)$. Hence, 
$$
\E_0\big(\1_{\{\omega(e)=1\}} \log \pi_\omega(0,e) + f(\omega, e)| \mathcal B_{k-1}\big) + G_{k,\eps}(\omega,e) 
$$
converges weakly in $L^p(\P_0)$ for any $p>1$ to $\1_{\{\omega(e)=1\}} \log \pi_\omega(0,e) + f(\omega, e)+ G_\eps(\cdot,e)$. By Mazur's theorem, there exists a finite convex combination of the above terms which converges strongly to $\1_{\{\omega(e)=1\}} \log \pi_\omega(0,e) + f(\omega, e)+ G_\eps(\cdot,e)$ in $L^p(\P_0)$ and hence, along a further subsequence to $\1_{\{\omega(e)=1\}} \log \pi_\omega(0,e) + f(\omega, e)+ G_\eps(\cdot,e)$, 
$\P_0$- almost surely. 

We take conditional expectation on both sides of \eqref{equiv4} with respect to $\mathcal B_{k-1}$ and use Jensen's inequality to get
$$
\log \sum_{e\in\mathbb U_d} \exp \big\{ \1_{\{\omega(e)=1\}} \log \pi_\omega(0,e) + \E_0\big(f(\omega, e)| \mathcal B_{k-1}\big) + G_{k,\eps}(\omega,e) \big\} \leq \overline H(f) + \eps.
$$
Again, applying Jensen's inequality to the aforementioned convex combination, and subsequently taking $k\to \infty$, we get
$$
\log \sum_{e\in\mathbb U_d} \1_{\{\omega(e)=1\}} \log \pi_\omega(0,e) \exp \big\{f(\omega, e)+ G_{\eps}(\omega,e) \big\} \leq \overline H(f) + \eps.
$$
This finishes the proof of Proposition \ref{propequal}.

 \end{proof}

\section{Large deviation bounds}
We start with a lemma which proves an important property of of the functional $\mathfrak I$ defined in \eqref{Idef}.
\begin{lemma}\label{convexI}
$\mathfrak I$ and $\mathfrak I^{\star\star}$ are convex in $\Mcal_1(\Omega_0 \times  \mathbb U_d)$. 
\end{lemma}
\begin{proof}
Fix $x\in (0,1)$ and $\mu, \nu \in \Mcal_1(\Omega_0 \times  \mathbb U_d)$. It is enough to show that if $\mu, \nu \in \Mcal_1^\star$, then $\mathfrak I(x \mu+ (1-x) \nu) \leq x \mathfrak I(\mu) + (1-x) \mathfrak I(\nu)$.

For $\mu,\nu\in\Mcal_1^\star$, we define,
$$
\begin{aligned}
(\tilde \pi_\mu, \phi_\mu)&= \bigg(\frac {\d\mu}{\d (\mu)_1}, \frac{\d (\mu)_1}{\d \P_0}\bigg), \\
(\tilde \pi_\nu, \phi_\nu)&= \bigg(\frac {\d\nu}{\d (\nu)_1}, \frac{\d (\nu)_1}{\d \P_0}\bigg).
\end{aligned}
$$
Then, by Lemma \ref{onetoone}, $(\tilde \pi_\mu, \phi_\mu), (\tilde \pi_\mu, \phi_\mu)\in \mathcal E$ and
$$
\begin{aligned}
\mathfrak I(\mu)= \int \d \P_0 \phi_\mu(\omega) \sum_e \tilde \pi_\mu(\omega,e) \, \log \frac{\tilde\pi_\mu(\omega,e)}{\pi_\omega(0,e)},\\
\mathfrak I(\nu)=\int \d \P_0 \phi_\nu(\omega) \sum_e \tilde \pi_\nu(\omega,e) \, \log \frac{\tilde\pi_\nu(\omega,e)}{\pi_\omega(0,e)}.
\end{aligned}
$$ 
Denote by $\lambda= x \mu+ (1-x) \nu$. Then
$(\tilde \pi, \phi)= \big(\frac {\d\lambda}{\d (\lambda)_1}, \frac{\d (\lambda)_1}{\d \P_0}\big)$ where
$$
\phi= x \phi_\mu+ (1-x) \phi_\nu , \qquad \tilde\pi= y \tilde \pi_\mu+ (1-y) \tilde \pi_\nu,
$$
with $y(\cdot)= \frac {x \phi_\mu(\cdot)} {x \phi_\mu(\cdot)+ (1-x) \phi_\nu(\cdot)}$. Then,
$$
\begin{aligned}
\mathfrak I(\lambda)&= \mathfrak I(x \mu+ (1-x) \nu)\\
&=\int \d \P_0 \phi(\omega) \sum_e \tilde \pi(\omega,e) \, \log \frac{\tilde\pi(\omega,e)}{\pi_\omega(0,e)}\\
&\leq x \bigg\{\int \d \P_0 \phi_\mu(\omega) \sum_e \tilde \pi_\mu(\omega,e) \, \log \frac{\tilde\pi_\mu(\omega,e)}{\pi_\omega(0,e)}\bigg\}+ (1-x) \bigg\{\int \d \P_0 \phi_\nu(\omega) \sum_e \tilde \pi_\nu(\omega,e) \, \log \frac{\tilde\pi_\nu(\omega,e)}{\pi_\omega(0,e)}\bigg\}\\
&= x \mathfrak I(\mu)+ (1-x) \mathfrak I(\nu),
\end{aligned}
$$
where we used convexity of the function $f\mapsto f\log f$ and Jensen's inequality for the upper bound.
\end{proof}
We end with the proof of Theorem \ref{thmlevel2} and Corollary \ref{thmlevel1}.

{\it{Proof of Theorem \ref{thmlevel2}:}} By Theorem \ref{thmmomgen}, 
$$
\lim_{n\to\infty} \frac 1n \log E^{\pi,\omega}_0 \big\{\exp\big\{n \langle f, \mathcal L_n\rangle\big\}\big\} = \sup_{\mu\in\Mcal_1^\star} \big\{ \langle f, \mu\rangle- \mathfrak I(\mu)\big\}= \sup_{\mu\in\Mcal_1(\Omega_0 \times  \mathbb U_d)} \big\{ \langle f, \mu\rangle- \mathfrak I(\mu)\big\}= \mathfrak I^\star(\mu).
$$
Since $\Omega_0$ is a closed subset of $\Omega=\{0,1\}^{\mathbb B_d}$ and hence, is compact, $\Mcal_1(\Omega_0 \times  \mathbb U_d)$ is compact in the weak topology. The upper bound \eqref{ldpub} for all closed sets now follows from
Theorem 4.5.3 \cite{DZ98}. The lower bound \eqref{ldplb} has been proved by Lemma \ref{lemmalb}.
\qed

{\it{Proof of Corollary \ref{thmlevel1}:}} The claim follows by contraction principle 
once we show that $\inf_{\xi(\mu)=x} \mathfrak I(\mu)= \inf_{\xi(\mu)=x} \mathfrak I^{\star\star}(\mu)$. This is easy to check using
convexity of $\mathfrak I$ and $\mathfrak I^{\star\star}$.
\qed

We finally prove that $\mathfrak I$ is not lower semicontinuous. 

\begin{lemma}\label{nonlsc}
Let $d\geq 2$ and $p>p_c(d)$. Then $\mathfrak I$ is not lower-semicontinuous on $\Mcal_1(\Omega_0\times \mathbb U_d)$. Hence, $\mathfrak I\ne \mathfrak I^{\star\star}$.
\end{lemma}
\begin{proof}
For any $\beta>1$, we define
$$
\pi^{\ssup \beta}(\omega,e)= \frac{\psi(e) \1_{\{\omega(e)=1\}}} {\sum_{e^\prime\in \mathbb U_d}\psi(e^\prime) \1_{\{\omega(e^\prime)=1\}}} \in \widetilde\Pi,
$$
where
$$
\psi(e)=
\begin{cases}
\beta>1 \qquad\mbox{if } e=e_1,\\
1 \qquad\qquad\mbox{else}.
\end{cases}
$$
Let $X_n^{\ssup \beta}$ be the Markov chain with transition probabilities $\pi^{\ssup\beta}$. By \cite{BGP02} and \cite{Sz02}, there exists $\beta_u=\beta_u(p,d)>0$ so that for $\beta>\beta_u$, the limiting speed 
$$
\lim_{n\to\infty} \frac{X_n^{\ssup \beta}}{n},
$$
 which exists and is an almost sure constant (see \cite{BGP02} and \cite{Sz02}), is zero. Then, by Kesten's lemma (see \cite{k75}),
there exists no $\phi\in L^1(\P_0)$ so that $(\pi^{\ssup\beta},\phi)\in \mathcal E$. We split the proof into two cases.

Suppose there exists a neighborhood $\mathfrak u$ of $\pi^{\ssup\beta}$ so that every $\tilde\pi^{\ssup\beta} \in \overline {\mathfrak u}$ fails to have an invariant density. Then, 
for any $\tilde\pi^{\ssup\beta} \in \overline {\mathfrak u}$ and any probability density $\phi\in L^1(\P_0)$, let $\mu_\beta$ be the corresponding element in $\Mcal_1(\Omega_0\times\mathbb U_d)$ (i.e., 
$\d\mu_\beta(\omega,e)=\pi^{\ssup\beta}(\omega,e)\phi(\omega)\d\P_0(\omega)$). Since 
$$(\tilde\pi^{\ssup\beta},\phi)\notin\mathcal E,
$$
by Lemma \ref{onetoone}, $\mu_\beta\notin   \Mcal_1^\star$. Then, $\mathfrak I(\mu_\beta)=\infty$ by \eqref{Idef}. If $\mathfrak I$ were lower semicontinuous on $\Mcal_1(\Omega_0\times \mathbb U_d)$, then $\mathfrak I=\mathfrak I^{\star\star}$ and by Theorem \ref{thmlevel2},
\begin{equation}\label{eq:superexp}
P^{\pi,\omega}_0\big\{\mathcal L_n \in \mathfrak n\big\}
\end{equation}
would decay super-exponentially for $\P_0$- almost every $\omega\in \Omega_0$, with $\mathfrak n$ being some neighborhood of $\mu_\beta$. However, since for every $\omega$, the relative entropy of $\pi^{\ssup\beta}(\omega,\cdot)$ w.r.t. $\pi_\omega(0,\cdot)$ is bounded below and above, the probability in \eqref{eq:superexp} decays exponentially and we have a contradiction.

Assume that there exists no such neighborhood $\mathfrak u$ of $\pi^{\ssup\beta}$. Let $\tilde\pi_n\to\pi^{\ssup\beta}$ such that for all $n\in \N$, $\tilde\pi_n$ has an invariant density $\phi_n$ and $(\tilde\pi_n,\phi_n)\in\mathcal E$. If $(\mu_n)_n$ is the sequence corresponding to $(\tilde\pi_n,\phi_n)$, since $\Mcal_1(\Omega_0\times\mathbb U_d)$ is compact, $\mu_n\Rightarrow\mu_\beta$ weakly along a subsequence. However, by our choice of $\beta>\beta_u$, $(\pi^{\ssup\beta},\phi)\notin \mathcal E$ for any density $\phi$ and hence $\mu_\beta\notin \Mcal_1^\star$ and $\mathfrak I(\mu_\beta)=\infty$. But,
$$
\lim_{n\to\infty} \mathfrak I(\mu_n)= \int \d\P_0 \phi(\omega) \sum_{e\in \mathbb U_d} \pi^{\ssup\beta}(\omega,e) \log \frac{\pi^{\ssup\beta}(\omega,e)}{\pi_\omega(0,e)},
$$
which is clearly finite. This proves that $\mathfrak I$ is not lower semicontinuous. 
\end{proof}

{\bf{Acknowledgments.}} Both authors were partially supported by ERC StG grant 239990. The second author wishes to thank S. R. S. Varadhan (New York)
for his encouragement and valuable discussions. He also thanks Nina Gantert (Munich) for her support.





\end{document}